\documentclass[12pt]{amsart}
\usepackage{amssymb}
\usepackage{graphicx}
\usepackage{color}
\usepackage{latexsym}
\usepackage{amssymb,amsmath,amstext}
\usepackage{epsfig}
\usepackage{graphics}
\usepackage{amsthm}
\usepackage{verbatim}
\usepackage{placeins}
\usepackage[colorlinks=true, urlcolor=blue, linkcolor=blue, citecolor=blue]{hyperref}
\usepackage[margin=1in]{geometry}
\numberwithin{equation}{section}
\theoremstyle{plain}%
\newtheorem{theorem}{Theorem}
\numberwithin{theorem}{section}
\newtheorem{proposition}[theorem]{Proposition}
\newtheorem{example}[theorem]{Example}
\newtheorem{lemma}[theorem]{Lemma}
\newtheorem{corollary}[theorem]{Corollary}
\newtheorem{definition}[theorem]{Definition}
\newtheorem{remark}[theorem]{Remark}

\newcommand{\C}{\mathbb{C}}

\allowdisplaybreaks

\begin{document}

\title{\bf Higher-Distance Commuting Varieties}

    \author{Madeleine Elyze, Alexander Guterman$^{1,2,3}$, Ralph Morrison$^{4}$, and Klemen \v{S}ivic$^{5}$
}

\date{1 October, 2020. \newline $^{1}$ Lomonosov Moscow State University, Moscow, 119991, Russia. \newline     $^{2}$ Moscow Institute of Physics and Technology, Dolgoprudny, 141701, Russia.
\newline $^{3}$  Moscow Center for Fundamental and Applied Mathematics, Moscow, 119991, Russia.
\newline     $^{4}$ Williams College, Williamstown, MA 01267 USA
\newline $^{5}$ University of Ljubljana, Faculty of mathematics and physics, Jadranska 19, SI-1000 Ljubljana, Slovenia}

\maketitle

\begin{abstract}\noindent The commuting variety  of matrices over a given field is a well-studied object in linear algebra and algebraic geometry.  As a set, it consists of all pairs of square matrices with entries in that field that commute with one another.  In this paper we generalise the commuting variety by using the commuting distance of matrices.  We show that over an algebraically closed field, each of our sets does indeed form a variety.  We compute the dimension of the distance-$2$ commuting variety and characterize its irreducible components. We also work over other fields, showing that the distance-$2$ commuting set is a variety but that the higher distance commuting sets may or may not be varieties, depending on the field and on the size of the matrices.
\end{abstract}


\medskip


\medskip


\section{Introduction}

Let $k$ be a field,  $n\geq 2$ an integer, and $\text{Mat}_{n\times n}(k)$ the set of all $n\times n$ matrices over $k$; sometimes we will simply write $\text{Mat}_{n\times n}$ when $k$ is implicit. We are interested in those pairs of matrices $A,B\in\text{Mat}_{n\times n}$ which commute with one another under matrix multiplication.   For instance, if a matrix is \emph{scalar}, that is, a constant multiple of the identity matrix, then all matrices commute with it.  One fruitful approach for studying all pairs of commuting matrices is to study the \emph{$n\times n$ commuting variety} \cite{GS, MT}.  We identify $\text{Mat}_{n\times n}^2$ with the $2n^2$-dimensional space $k^{2n^2}$, so that a pair of matrices $(A,B)\in\text{Mat}_{n\times n}^2$ corresponds to a point in $k^{2n^2}$.  As a set, the commuting variety consists of all points $(A,B)\in k^{2n^2}$ such that $AB=BA$.  It has the structure of an \emph{affine variety}, meaning that it is the solution set of a finite collection of polynomial equations.  In particular, it is defined by the $n^2$ polynomial equations of the form $(XY-YX)_{ij}=0$, where $X$ and $Y$ are $n\times n$ matrices filled with variables $x_{ij}$ and~$y_{ij}$.  This variety is known to be irreducible \cite{gur, MT}, meaning that it cannot be written as the union of two 
blue proper subvarieties.  This study of commutativity of matrices through varieties  can be generalised to varieties of commuting triples of matrices \cite{GS}, or commuting $p$-tuples of matrices.

In this paper we introduce other generalizations of the commuting matrix variety, and prove that these generalizations are indeed affine varieties.  To define them, we need a measure of how close two matrices are to commuting with one another.  We will use the well-studied \emph{commuting distance} of matrices, which are defined using the \emph{$n\times n$ commuting graph} $\Gamma(\text{Mat}_{n\times n})$ over $k$ \cite{AGHM}. This graph has vertices corresponding to the $n\times n$ non-scalar matrices over $k$, with an edge between $A$ and $B$ if and only if $AB=BA$.   With a view towards defining {distance}, a natural question to ask is whether $\Gamma(\text{Mat}_{n\times n})$ is connected, and if so what is its \emph{diameter}; that is, what is the supremum of the length of all shortest paths between pairs of vertices.
The answer to this question depends both on the field $k$ and on the integer $n$.  As noted in \cite[Remark 8]{AR} and discussed in Proposition \ref{prop:n=2}, for $n=2$ the graph $\Gamma(\text{Mat}_{n\times n})$ is disconnected.  However, it was shown in \cite{AMRR} that if $n\geq 3$ and $k$ is an algebraically closed field, then $\Gamma(\text{Mat}_{n\times n})$ is connected, and in fact $\text{diam}(\Gamma(\text{Mat}_{n\times n}))=4$.  The same holds when $k=\mathbb{R}$ \cite{shitov}.  Note that the situation is quite  different for other fields. For example, the graph $\Gamma(\text{Mat}_{n\times n})$ over the rationals ${\mathbb Q}$ is not connected \cite[Remark 8]{abm}.

Whenever $\Gamma(\textrm{Mat}_{n\times n}(k))$ is connected, for any two non-scalar matrices $A,B\in \text{Mat}_{n\times n}$ we can define $d_k(A,B)$ to be the distance between 
the corresponding vertices on the graph $\Gamma(\text{Mat}_{n\times n}(k))$; we will drop the subscript $k$ and simply write $d(A,B)$ when the field is implicit.  Under this definition, $d(A,B)=0$ if and only if $A=B$, and $d(A,B)=1$ if and only if $A\neq B$ and $A$ and $B$ commute.  We extend the function $d(\cdot\,,\cdot)$ to take \emph{any} two $n\times n$ matrices as input by letting $d(A,B)=1$ whenever $A$ or $B$ is a scalar matrix (unless $A=B$, in which case $d(A,B)=0$).  If $\Gamma(\text{Mat}_{n\times n})$ is not connected, then we take $d(A,B)=\infty$ whenever $A$ and $B$ are nonscalar matrices in different components of the graph.

We define the \emph{distance-$d$ commuting set} $\mathcal{C}^{d}_n(k)$ to be the set of pairs of $n\times n$ matrices with entries in $k$ that have commuting distance at most $d$.  That is,
 $$\mathcal{C}^{d}_n(k):=\{(A,B)\,|\, d(A,B)\leq d\}\subset \text{Mat}_{n\times n}^2.$$
When $k$ is clear from context, we  abbreviate this notation to $\mathcal{C}^{d}_n$. Perhaps the most familiar of these sets is $\mathcal{C}^{1}_n$, which is simply the $n\times n$ commuting variety.  Our main theorem shows that regardless of our choice of $d$, the set $\mathcal{C}^{d}_n$ is still an affine variety, at least over a field that is algebraically closed.

\begin{theorem}\label{theorem:main} Let $k$ be an algebraically closed field.  Then for any $d\geq 0$, the set $\mathcal{C}^{d}_n$ is an affine variety in the $2n^2$-dimensional affine space $\text{Mat}^2_{n\times n}$.  The same result holds for any field if $d\leq 2$.
\end{theorem}

There do exist choices of $k$, $n$, and $d$ such that $\mathcal{C}^{d}_n(k)$ is not a variety, as illustrated in the following proposition suggested to the authors by Jan Draisma.

\begin{proposition}\label{prop:d_geq_4}
Let $n\geq 3$ and $d\geq 4$.  The set $\mathcal{C}_n^d(k)$ is a variety if and only if $k$ is finite or $\textrm{diam}(\Gamma(\textrm{Mat}_{n\times n}(k)))\leq d$.
\end{proposition}

The final remaining case of $d=3$ for infinite, non-algebraically closed fields is more subtle.  It turns out that such commuting sets are not always varieties.

\begin{proposition}\label{prop:c34r}
The distance-$3$ commuting set of pairs of $4\times 4$ real matrices $\mathcal{C}_4^3(\mathbb{R})$ is not a variety.
\end{proposition}





The fact that $\mathcal{C}^{d}_n$ is a variety over many fields has a number of nice consequences.  For instance, every variety has a well-defined notion of dimension, meaning that this set is a reasonable geometric object.  Moreover, if we are working over $k=\mathbb{C}$, or $k=\mathbb{R}$ with $d\leq 2$, every variety is closed in the usual Euclidean topology, so we know that these sets are closed.

It turns out that many instances of Theorem \ref{theorem:main} are readily proven.  In the case of $d=0$, we have that $\mathcal{C}^{0}_n=\{(A,B)\,|\, A=B\}$, which is defined by the $n^2$ equations of the form $(X-Y)_{ij}$=0, and is thus an affine variety.  As already noted, $\mathcal{C}^{1}_n$ is the commuting variety.  In the case of $n=2$ and $d\geq 2$, any two matrices $A$ and $B$ with $d(A,B)\leq d$ must in fact satisfy $d(A,B)\leq 1$ (see Proposition \ref{prop:n=2}), so $\mathcal{C}^{d}_n=\mathcal{C}^{ 1}_n$.  Since $\mathcal{C}^1_n$ is a variety, the theorem holds in this case.  Finally, if $n\geq 3$ and $d\geq 4$ then we have
$$\mathcal{C}_n^{d}=\mathcal{C}_n^{ 4}=\text{Mat}_{n\times n}^2$$
since $d(A,B)\leq 4$ for all $A$ and $B$ \cite{AMRR}, \cite{shitov}.  This set is indeed a variety: it is the vanishing locus of the zero polynomial.

Thus, it remains to show that Theorem \ref{theorem:main} holds in the case where $n\geq 3$, and either $d=2$ or $d=3$; and to prove our results for more general fields.  After presenting useful background material on varieties and on commuting distance in Section \ref{section:background} (as well as proving Proposition \ref{prop:d_geq_4}), we handle the $d=2$ case in Section \ref{section:distance2}.  We also give an explicit decomposition of the distance-$2$ commuting variety into irreducibles, and compute its dimension. The $d=3$ case, for both algebraically closed fields and otherwise, is studied in Section \ref{section:distance3}.

\section{Affine varieties and commuting distance}
\label{section:background}

\subsection{Affine varieties}
We devote the first part of this section to present the tools and results from algebraic geometry that we will use.  Most of our notation and terminology comes from \cite{CLO}. Let $k$ be a field, and let $R=k[x_1,\ldots,x_n]$ be the polynomial ring in $n$ variables over $k$.  We refer to $k^n$ as \emph{$n$-dimensional affine space}. Let $f_1,\ldots,f_s\in R$.  The \emph{vanishing locus} of $f_1,\ldots,f_s$ is the set of points in $k^n$ at which all the polynomials vanish:
$$\textbf{V}_k(f_1,\ldots,f_s):=\{(a_1,\ldots,a_n)\in k^n\,|\, f_i(a_1,\ldots,a_n)=0\text{ for $1\leq i\leq s$}\}.$$
We refer to any set of this form as an \emph{affine variety}.  Given an ideal $I\subset R$, we define its vanishing locus as
$$\textbf{V}_k(I):=\{(a_1,\ldots,a_n)\in k^n\,|\, f(a_1,\ldots,a_n)=0\text{ for all $f\in I$}\}.$$
If $I=\left<f_1,\ldots,f_s\right>$, then we have $\textbf{V}_k(I)=\textbf{V}_k(f_1,\ldots,f_s)$ \cite[Proposition 2.5.9]{CLO}.  Since every ideal $I\subset R$ can be generated by finitely many polynomials \cite{hilbert}, we could equivalently define an affine variety to be any set of the form $\textbf{V}_k(I)$ where $I$ is an ideal in  $R$.

\begin{example}\label{example:2x2}  We consider $\text{Mat}_{2\times 2}^2$ as  $8$-dimensional affine space, each point of which corresponds to a pair of $2\times 2$ matrices $(A,B)=\left(\left(\begin{smallmatrix}a_{11}&a_{12}\\a_{21}&a_{22}\end{smallmatrix}\right),\left(\begin{smallmatrix}b_{11}&b_{12}\\b_{21}&b_{22}\end{smallmatrix}\right)\right)$.  Our ring of polynomials can then be written as $R=k[x_{11},x_{12},x_{21},x_{22},y_{11},y_{12},y_{21},y_{22}]$, where the variable $x_{ij}$ (respectively $y_{ij}$) corresponds to the coordinate $a_{ij}$ (respectively $b_{ij}$).  For shorthand, we could also write $R=k[X,Y]$, where $X=\left(\begin{smallmatrix}x_{11}&x_{12}\\x_{21}&x_{22}\end{smallmatrix}\right)$ and $Y=\left(\begin{smallmatrix}y_{11}&y_{12}\\y_{21}&y_{22}\end{smallmatrix}\right)$.  Consider the variety $V(f_1,f_2,f_3,f_4)$, where
\begin{align*}
f_1=&x_{11}y_{11}+x_{12}y_{21}-x_{11}y_{11}-x_{21}y_{12}
\\f_2=&x_{11}y_{12}+x_{12}y_{22}-x_{12}y_{11}-x_{22}y_{12}
\\f_3=&x_{21}y_{11}+x_{22}y_{21}-x_{11}y_{21}-x_{21}y_{22}
\\f_4=&x_{21}y_{12}+x_{22}y_{22}-x_{12}y_{21}-x_{22}y_{22}.
\end{align*}
These polynomials are simply the entries of $XY-YX$.  It follows that $f_1(A,B)=f_2(A,B)=f_3(A,B)=f_4(A,B)=0$ if and only if $AB-BA=0$, that is if and only if $A$ and $B$ commute.  Thus, $\textbf{V}_k(f_1,f_2,f_3,f_4)$ is the $2\times 2$ commuting variety.  In fact, since $f_1=-f_4$, we have $\textbf{V}_k(f_1,f_2,f_3,f_4)=\textbf{V}_k(f_1,f_2,f_3)$.

\end{example}

If a set $S\subset k^n$ is an affine variety, we say that it is \emph{Zariski-closed}.  This terminology comes from the \emph{Zariski topology} on $k^n$, which defines a subset of $k^n$ to be closed if it is equal to $\textbf{V}_k(f_1,\ldots,f_s)$ for some $f_1,\ldots,f_s\in k[x_1,\ldots,x_n]$.  This defines a topology because the union of finitely many varieties is a variety; because the intersection of any collection of varieties is a variety; and because $\emptyset=\textbf{V}_k(1)$ and $k^n=\textbf{V}_k(0)$ are both varieties. 
\begin{definition}
	For a general set $S\subset k^n$, the \emph{Zariski closure} of $S$, denoted $\overline{S}$, is the smallest Zariski-closed set (that is, the smallest variety) containing $S$.
\end{definition}
\begin{example} \label{eq:S}  Consider the set
$$S=\{(a,a)\,|\,a\neq 0\}\subset \mathbb{R}^2.$$
This set is not a variety:  if a polynomial $f(x,y)$ vanishes on the set $S$, then since polynomials are continuous functions we  also have $f(0,0)=0$.  It follows that any variety containing $S$ must also contain the point $(0,0)$. 

In fact, the Zariski closure of $S$ is simply $S$ together with this point:
$$\overline{S}=S\cup\{(0,0)\}=\textbf{V}_\mathbb{R}(x-y).$$
\end{example}

 In the previous example over $\mathbb{R}$, the Zariski closure of our set happened to equal its closure in the usual Euclidean topology.  When working over $\mathbb{R}$ or $\mathbb{C}$, since polynomials are continuous functions, all Zariski-closed sets are closed in the Euclidean topology, so it is true that the Zariski closure always contains the Euclidean closure.  In general, this containment is strict, although in certain special cases we have equality, such as for \emph{constructible} sets.

\begin{definition}
We say that a set in $k^n$ is \emph{constructible} if it is defined by a finite boolean combination of polynomial conditions.  In other words, it can be written as a finite combination of intersections and unions of affine varieties and their complements.\end{definition}  

For example, the set $S$ from Example \ref{eq:S} is constructible, since $S=\textbf{V}_\mathbb{R}(x-y)\setminus\textbf{V}_{\mathbb{R}}(x,y)$.

\begin{proposition}[Corollary 1 in Section 1.10 of \cite{mum}]\label{proposition:closure}  Suppose $S\subset\mathbb{C}^n$ is a constructible set.  Then the Zariski closure of $S$ is equal to the Euclidean closure of $S$.
\end{proposition}

A variety $V$ is called \emph{irreducible} if it is irreducible in the Zariski topology, i.e. if it cannot be written as a union $V=V_1\cup V_2$ of two 
blue proper subvarieties $V_1\subsetneq V$ and $V_2\subsetneq V$.  Any variety $V$ can be written as a finite union of irreducible varieties $V_1,\ldots,V_s$, unique up to reordering with no $V_i$ contained in any $V_j$ for $i\neq j$ \cite[Theorem 4.6.4]{CLO}.

The \emph{ideal-variety correspondence} \cite[Chapter 4]{CLO} allows us to connect geometric operations on affine varieties to algebraic operations on ideals.  Often these correspondences only hold up to taking Zariski closures.

\begin{proposition}[Theorem 3.2.3 in \cite{CLO}] \label{clo:projection} Let $k$ be an algebraically closed field, and let $I\subset k[x_1,\ldots,x_n]$ be an ideal.  Let $1\leq m\leq n$, and let $\pi:k^n\rightarrow k^m$ denote projection onto the first $m$ coordinates.  Then
$$\overline{\pi(\textbf{V}_k(I))}={\textbf{V}_k(I\cap k[x_1,\ldots,x_m])}.$$
\end{proposition}
In other words, up to Zariski closure, projecting a variety corresponds to eliminating variables.

\begin{proposition}[Theorem 4.4.10 in \cite{CLO}] \label{clo:difference} Let $k$ be an algebraically closed field, and let $I,J\subset k[x_1,\ldots,x_n]$ be ideals.  Let
$$I:J^\infty=\{f\in k[x_1,\ldots,x_n]\,|\,\text{for all $g\in J$, there is $N\geq 0$ such that $fg^N\in I$}\}.$$
Then $I:J^\infty$ is an ideal, and
$$\textbf{V}_k(I:J^\infty)=\overline{\textbf{V}_k(I)\setminus\textbf{V}_k(J)}.$$
\end{proposition}
\begin{definition}
The ideal $I:J^\infty$ is called the \emph{saturation} of $I$ with respect to $J$.	
\end{definition}
So, up to Zariski closure, taking a set-theoretic difference of varieties corresponds to saturating one ideal 
with respect to another.

We now build up some tools that will help us study the distance-$d$ commuting sets. We briefly  discuss \emph{projective algebraic geometry} \cite[Chapter 8]{CLO}.  Although our main results are stated in the language of affine varieties, projective geometry will be used in the proofs of Theorem \ref{theorem:decomposition} and Proposition \ref{prop:pcaffine}.  Again we choose a field $k$ and an integer $n$.  The role of affine space is now played by projective space $\mathbb{P}_k^n$, which as a set is all equivalence classes of $(n+1)$-tupes $(a_0:a_1:\cdots:a_n)$ where $a_i\in k$, not all equal to $0$.  The equivalence is defined by $(a_0:a_1:\cdots:a_n)\sim (b_0:b_1:\cdots:b_n)$ if and only if there exists a nonzero constant $\lambda$ with $a_i=\lambda b_i$ for all $i$.  There are a number of other constructions of projective space, such as the one given in example below.

\begin{example} Let $V$ be a non-zero vector space.  Then the \emph{projectivization of $V$}, denoted $\mathbb{P}(V)$, is the set of one-dimensional vector subspaces of $V$.
\end{example}

In order to talk about polynomials \emph{vanishing} at a point of projective space, we must consider polynomials that are \emph{homogeneous}; that is, polynomials where all terms have equal degree.  This is because if $f\in k[x_0,\ldots,x_n]$ is homogeneous with each term having degree $d$, then
$$f(\lambda a_0,\lambda a_1,\ldots,\lambda a_n)=\lambda^df(a_0,a_1,\ldots,a_n)$$
for any $a_i\in k$ and any $\lambda$.  Thus, $f$ vanishing at the tuple $(a_0:\cdots:a_n)$ implies it vanishes for all other tuples in the same equivalence class.

We can consider a product of a projective space and an affine space:
$$k^m\times \mathbb{P}^n_k.$$
A subset of such a space is Zariski-closed if and only if it is the vanishing locus of a collection of polynomials, which are homogeneous in the variables corresponding to the coordinates of $\mathbb{P}^n_k$.  A key result that we will use is that when projecting away from projective spaces, Zariski-closed sets are mapped to Zariski-closed sets, at least over an algebraically closed field.

\begin{proposition}[Proposition 8.5.5 and Theorem 8.5.6 in \cite{CLO}]\label{proposition:proper}
Let $k$ be an algebraically closed field, and let $\pi:k^m\times \mathbb{P}^n_k\rightarrow k^m$ be the projection map onto the affine space.  Then $\pi$ maps Zariski-closed sets to Zariski-closed sets.
\end{proposition}

In the language of morphisms, such a projection $\pi$ is a \emph{proper} morphism.  We can generalise this to have more copies of affine and projective spaces, and the result still holds as long as we are projecting away from projective spaces.

Note that this result does not hold when projecting away from affine spaces rather than projective spaces:  the projection map $\pi:\mathbb{C}\times \mathbb{C}\rightarrow \mathbb{C}$ defined by $\pi(a,b)=a$ sends the hyperbola $\textbf{V}_\mathbb{C}(xy-1)$ to the set $\{a\in\mathbb{C}\,|\, a\neq 0\}$, which is not a variety.  Moreover, an analog of Proposition \ref{proposition:proper} does not hold for the field of real numbers, as the following example shows.
\begin{example}
Let $S=\Bigl\{([x_0:x_1],y)\vert x_0^2=yx_1^2\Bigr\}\subset {\mathbb P}^1({\mathbb R}) \times {\mathbb R}$. Then $S$ is Zariski-closed, but the projection to the second component is the set of non-negative reals, which is not Zariski-closed.
\end{example}


Several of our results in Section \ref{section:distance3} will study varieties over the real numbers.  To help with that  endeavor we close this subsection with the following result.  

\begin{lemma}\label{lemma:realclosed}
Let $k$ be a field, and let $K$ be an algebraic extension of $k$.  If $V\subset K^n$ is an affine variety in $K^n$, then $V\cap k^n$ is a variety in $k^n$
\end{lemma}

\begin{proof}
Let $V=\textbf{V}_{K}(f_1,\ldots,f_r)$ be an affine variety in $K^n$, where $f_i\in K[x_1,\ldots,x_n]$ for all $i$.  Since $K$ is an algebraic extension of $k$, there exists a finite extension $k'$ of $k$ containing all coefficients of all $f_i$.  Note that $W:=\textbf{V}_{k'}(f_1,\ldots,f_r)=\textbf{V}_{K}(f_1,\ldots,f_r)\cap (k')^n$, since both sets are the set of all points in $(k')^n$ where $f_1,\ldots,f_r$ simultaneously vanish.  

Let $\alpha_1,\ldots,\alpha_s$ be a vector space basis for $k'$ over $k$.  We can write each $f_i$ as
\[f_i=\sum_{j=1}^s \alpha_jg_{ij},\]
where $g_{ij}\in k[x_1,\ldots,x_n]$ for all $i,j$.  We claim that $W\cap k^n=\textbf{V}_k(g_{11},\ldots,g_{rs})$.  Certainly   $W\cap k^n\supset \textbf{V}_k(g_{11},\ldots,g_{rs})$:  if all $g_{ij}$ vanish at a point, then all $f_i$ do the same. 
For the other direction, suppose $p\in W\cap k^n$, so that $f_i(p)=0$ for all $i$.  This means that 
\[\sum_{j=1}^s \alpha_jg_{ij}(p)=0\]
for all $i$.  Since $\alpha_1,\ldots,\alpha_s$ is a vector space basis for $k'$ over $k$, this implies that $g_{ij}(p)=0$ for all $i$ and $j$.  Thus $p\in \textbf{V}_k(g_{11},\ldots,g_{rs})$, giving us the other direction of containment. This proves that $W\cap k^n$ is a variety; since $W\cap k^n=(V\cap (k')^n)\cap k^n=V\cap k^n$, this completes the proof.
\end{proof}

\subsection{Commuting distance} 
Now we present and develop material about commuting distance and commuting graphs.  
Let $A,B\in \text{Mat}_{n\times n}$.  We say that $A$ and $B$ commute if $AB=BA$.  For notation, we will sometimes write $A\leftrightarrow B$ to indicate that $A$ and $B$ commute.

Given two non-scalar matrices $A$ and $B$, define the \emph{commuting distance} $d(A,B)$ to be the distance between two matrices $A$ and $B$ on the commuting graph $\Gamma(\text{Mat}_{n\times n})$, so that $d(A,B)=0$ if and only if $A=B$, and $d(A,B)=1$ if and only if $A\neq B$ and $A$ and $B$ commute.  If $A$ and $B$ do not commute, then $d(A,B)=d$, where $d$ is the smallest natural number, if it exists, such that there exist $d-1$ non-scalar matrices $C_1,\ldots, C_{d-1}$ with $d(A,C_1)=d(C_{d-1},B)=d(C_i,C_{i+1})=1$ for $1\leq i\leq d-2$.  Such a commuting chain could be written as
$$A\leftrightarrow C_1 \leftrightarrow C_2 \leftrightarrow \cdots \leftrightarrow C_{d-1} \leftrightarrow B.$$  If there does not exist such a chain, we define $d(A,B)=\infty$.  We note that it is vital to require that all 
$C_i$ are non-scalar:  otherwise, we would have $d(A,B)\leq 2$ for all $A$ and $B$.
We extend the definition of the distance function  to allow scalar matrices as input by defining $d(A,A)=0$ and  $d(A,B)=d(B,A)=1$ whenever $A$ is a scalar matrix and $A\neq B$.

\begin{example}  Consider the following three matrices  over $\mathbb{C}$:
$$A=\begin{pmatrix}
1 & 2 & 0\\
3 & 4 & 0\\
0 & 0 & 5
\end{pmatrix},\,\,\,
B=\begin{pmatrix}
1 & 1 & 0\\
2 & 2 & 0\\
0 & 0 & 3
\end{pmatrix},\,\,\,
C=\begin{pmatrix}
1 & 0 & 0\\
0 & 1 & 0\\
0 & 0 & 0
\end{pmatrix}.
$$
Direct computation shows that $A\leftrightarrow C$ and $B\leftrightarrow C$, so $d(A,C)=1$ and $d(B,C)=1$. Similarly we can verify that $AB\neq BA$.  Since $A$ and $B$ do not commute but they do commute with the common non-scalar matrix $C$, we have $d(A,B)=2$.
\end{example}

\begin{proposition} \label{prop:n=2}  For any field $k$ let $A,B\in \text{Mat}_{2\times 2}$. Then either $d(A,B)\leq 1$ or $d(A,B)=\infty$.
\end{proposition}

\begin{proof}
 It will suffice to show that there do not exist any matrices $A$ and $B$ with $d(A,B)=2$:  if two matrices are separated by a commuting chain of length at least $2$, a subchain will yield two matrices of distance exactly $2$.  

Suppose for the sake of contradiction that there exist $A$ and $B$ with $d(A,B)=2$.  Let $C$ be a non-scalar matrix such that $A\leftrightarrow C\leftrightarrow B$.  Since $C$ is not a scalar matrix, it is not the root of a polynomial of degree $1$, meaning that the minimal polynomial of $C$ has degree at least $2$. The minimal polynomial thus coincides with 
the characteristic polynomial of $C$ up to a scalar factor.  Since $C$ is a matrix with equal minimal and characteristic polynomials, the matrices commuting with $C$ are precisely those of the form $p(C)$, where $p\in k[x]$ is a polynomial \cite[Theorem 3.2.4.2]{HJ}.  It follows that $A=p(C)$ and $B=q(C)$ for some polynomials $p$ and $q$.  But $p(C)$ and $q(C)$ commute, so $d(A,B)\leq 1$, a contradiction.
\end{proof}

Since over any field there are pairs of $2\times 2$ matrices that do not commute (for instance, $\left(\begin{smallmatrix}0&1\\0&0\end{smallmatrix}\right)$ and $\left(\begin{smallmatrix}0&0\\1&0\end{smallmatrix}\right)$),  we immediately get the following result originally noted in \cite[Remark 8]{AR}:
\begin{corollary}
The graph $\Gamma(\text{Mat}_{2\times 2})$ is disconnected, regardless of the field $k$.
\end{corollary}

Once $n\geq 3$, the connectedness of $\Gamma(\text{Mat}_{n\times n})$ is guaranteed by $k$ being algebraically closed or \emph{real closed}, which means that $\overline{k}/k$ is a nontrivial finite extension. The graph $\Gamma(\text{Mat}_{n\times n})$ can be disconnected in other cases, such as over the field ${\mathbb Q}$ of rational numbers  \cite[Remark~8]{abm}.

\begin{theorem}[Corollary 7 in \cite{AMRR}]  Let $k$ be an algebraically closed field, and let $n\geq 3$.  Then $\Gamma(\text{Mat}_{n\times n})$ is connected, and $\text{diam}(\Gamma(\text{Mat}_{n\times n}))=4$.
\end{theorem}
Below we provide a short proof of this fact for the completeness. 

To prove that the diameter is as most $4$, one constructs for any  $A$ and $B$ a commuting chain
$$A\leftrightarrow C\leftrightarrow D\leftrightarrow E\leftrightarrow B,$$
where $C$, $D$, and $E$ are nonscalar.
Choose $C$ and $E$ to be rank $1$ matrices built out of left and right eigenvectors of $A$ and of $B$, respectively.  Then choose $D$ to be a rank $1$ matrix commuting with both $C$ and $E$, which is possible for any two rank $1$ matrices as long as $n\geq 3$.  Thus, we have $d(A,B)\leq 4$ for all $A$ and $B$.  The same result for real closed fields is proved in~\cite{shitov}.

One way to see the diameter is indeed equal to $4$ (rather than to $2$ or $3$) is to note that an elementary Jordan matrix is always at distance four from its transpose, as proven in \cite{AMRR}.  The authors of \cite{DKO} go on to give a necessary condition for when two matrices can be at the maximal distance of four from one another.  A matrix $A$ is said to be \emph{non-derogatory} if  its characteristic polynomial is equal to its minimal polynomial.  Equivalently, a matrix is said to be non-derogatory if for each distinct eigenvalue $\lambda$ of $A$ there is only one Jordan block corresponding to $\lambda$ in the Jordan normal form for $A$
\cite[Chapter 7]{Cullen}.  If a matrix is not non-derogatory, we say it is \emph{derogatory}.

\begin{theorem}\label{theorem:equivalences} \cite[Theorem 1.1]{DKO}  Let $n\geq 3$ and $k$ be algebraically closed.  Then the following statements are equivalent for a non-scalar matrix $A\in \text{Mat}_{n\times n}$.
\begin{itemize}
\item[(i)]  $A$ is non-derogatory.
\item[(ii)] $A$ commutes only with elements of $k[A]$.
\item[(iii)]  There exists a matrix $B\in \text{Mat}_{n\times n}$ such that $d(A,B)=4$.
\end{itemize}
\end{theorem}

We close this section by proving Proposition \ref{prop:d_geq_4}.

\begin{proof}[Proof of Proposition \ref{prop:d_geq_4}]
In affine space over a finite field, any set is a variety, as it is the union of finitely many points.  Thus we will assume $k$ is infinite.

If $\textrm{diam}(\Gamma(\textrm{Mat}_{n\times n}(k)))\leq d$, then $\mathcal{C}_n^d(k)$ is all of $2n^2$-dimensional affine space, which is indeed a variety, namely $\textbf{V}(0)$.  Conversely, assume $\textrm{diam}(\Gamma(\textrm{Mat}_{n\times n}(k)))>d$.  This means that $\mathcal{C}_n^d(k)\subsetneq \textrm{Mat}_{n\times n}^2$.  Since $d\geq 4$, we know that $k$ is not algebraically closed.  We will show that the Zariski closure of $\mathcal{C}_n^4(k)$ is $\textrm{Mat}_{n\times n}^2$, implying that $\mathcal{C}_n^4(k)$ is not a variety.

Indeed, we first argue that the set of all pairs $(A,B)$, both diagonalizable over $k$, forms a Zariski dense set in $\textrm{Mat}_{n\times n}^2(k)$.  First we note that $k$ is Zariski dense in $\overline{k}$:  this is because any infinite subset of a field is Zariski dense in that field.  It follows that $\textrm{GL}_n(k)$ is Zariski dense in $\textrm{GL}_n({\overline{k}})$, and that $k^n$ is Zariski dense in ${\overline{k}}^n$.  This means that the image of the map
\begin{align*}  \textrm{GL}_n(k) \times k^n &\rightarrow \textrm{Mat}_{n\times n}(k)\\ (g,(\lambda_1,\cdots,\lambda_n)) &\mapsto g\, \textrm{diag}(\lambda_1,\cdots,\lambda_n) g^{-1}\end{align*}
is Zariski-dense in the image of the corresponding map over $\overline{k}$, which is Zariski dense in $\textrm{Mat}_{n\times n}\left(\overline{k}\right)$.  Thus the set of all $k$-diagonalizable matrices forms a Zariski dense subset of $\textrm{Mat}_{n\times n}\left({k}\right)$.

Let $(A,B)$ be a pair of matrices with entries in $k$, both diagonalizable over $k$; this means in particular that each has an eigenvalue in $k$, and so commutes with a rank $1$ matrix with entries in $k$.  If $C$ is a rank $1$ matrix commuting with $A$, and $D$ is a rank $1$ matrix commuting with $B$, then since $n\geq 3$ there exists a rank $1$ matrix $E$ commuting with both $C$ and $D$.
As all rank $1$ matrices are nonscalar, $d(A,B)\leq 4$.  As this holds on a Zariski dense subset of $\textrm{Mat}_{n\times n}^2(k)$, this means that $\overline{\mathcal{C}_n^4(k)}=\textrm{Mat}_{n\times n}^2(k)$.  Since $d\geq 4$ we have $\textrm{Mat}_{n\times n}^2(k)=\overline{\mathcal{C}_n^4(k)}\subset \overline{\mathcal{C}_n^d(k)}\subset \textrm{Mat}_{n\times n}^2(k)$, implying $\overline{\mathcal{C}_n^d(k)}=\textrm{Mat}_{n\times n}^2(k)$, thus completing the proof.
\end{proof}

\section{The distance-$2$ commuting set}
\label{section:distance2}

Fix $n\geq 3$, and let $k$ be any field.  In this section we will prove that $\mathcal{C}^{2}_n(k)$ is a variety.  In Proposition \ref{prop:minors}, we  show that it is the vanishing locus of certain minors of a $2n^2\times n^2$ matrix with linear polynomials as its entries.  We then present a more geometric construction of $\mathcal{C}^{2}_n$ as an affine variety over $\mathbb{C}$.  The proof that this construction actually gives $\mathcal{C}^{2}_n$, rather than some other set, relies on Proposition \ref{prop:minors}, and so does not constitute an independent proof that $\mathcal{C}^{2}_n$ is a variety.  
  We then find the irreducible decomposition of $\mathcal{C}_n^2$ over $\mathbb{C}$, as well as its dimension.

 \begin{proposition}\label{prop:minors}  Over any field, the set $\mathcal{C}^{2}_n$ is an affine variety.  In particular, it is defined by minors of a $2n^2\times n^2$ matrix with linear entries.
 \end{proposition}

 \begin{proof}  We will explicitly give the desired collection of polynomials that define the set $\mathcal{C}^{2}_n$.  Let $A,B,C$ be $n\times n$ matrices with entries given by $a_{ij}$, $b_{ij}$, and $c_{ij}$, respectively.  We have $A\leftrightarrow C$ if and only if
 $$\sum_{k=1}^na_{ik}c_{kj}-\sum_{\ell=1}^na_{\ell j}c_{i\ell}=0 $$
 for $1\leq i,j\leq n$.  Similarly, $B\leftrightarrow C$ if and only if
 $$\sum_{k=1}^nb_{ik}c_{kj}-\sum_{\ell=1}^nb_{\ell j}c_{i\ell}=0 $$
 for $1\leq i,j\leq n$.  Flattening the matrix $C$ into a vector $\textbf{c}$, we can write these equations together as
 $$M_{A,B}\textbf{c}=\textbf{0},$$
 where $M_{A,B}$ is an $2n^2\times n^2$ matrix with entries in $k[\{a_{ij},b_{ij}\}]$.  Thus $A\leftrightarrow C\leftrightarrow B$ if and only if the flattening of $C$ is in the nullspace of $M_{A,B}$.
 
 It follows that the set of matrices commuting with both $A$ and $B$ is a vector space (namely the vector space of solutions of the above equation) whose dimension is equal to the nullity of the matrix $M_{A,B}$.  The dimension of this vector space is at least $1$, since every scalar matrix commutes with all matrices.   Therefore, the nullity of $M_{A,B}$ is strictly greater than $1$ if and only if  $A$ and $B$  commute with a mutual non-scalar matrix, which is equivalent to  $d(A,B)\leq 2$.  By the rank-nullity theorem, $d(A,B)\leq 2$ if and only if the rank of $M_{A,B}$ is at most $n^2-2$.  This means that $(A,B)\in \mathcal{C}^{2}_n$ if and only if the $(n^2-1) \times (n^2-1)$ minors of $M_{A,B}$ are all equal to $0$.  These minors are polynomials in the entries of $A$ and $B$. Thus, the vanishing locus of these polynomials in $\text{Mat}_{n\times n}^2$ is precisely $\mathcal{C}^{2}_n$.  
 \end{proof}
 
 To see the techniques of this proof more explicitly, we construct the matrix $M_{A,B}$ for the case of $n=3$.
\begin{example}
Let $A,C,B\in\C^{3\times 3}$ such that
\[A=\begin{pmatrix}
a_{11} & a_{12} & a_{13}\\
a_{21} & a_{22} & a_{23}\\
a_{31} & a_{32} & a_{33}
\end{pmatrix}, \ B = \begin{pmatrix}
b_{11} & b_{12} & b_{13}\\
b_{21} & b_{22} & b_{23}\\
b_{31} & b_{32} & b_{33}
\end{pmatrix}, \ C = \begin{pmatrix}
c_{11} & c_{12} & c_{13}\\
c_{21} & c_{22} & c_{23}\\
c_{31} & c_{32} & c_{33}
\end{pmatrix}.\]

 Then $AC-CA=0$ if and only if the following nine equations are all satisfied: 
 \begin{align*}
  a_{13}c_{31} + a_{12}c_{21}+a_{11}c_{11}-c_{13}a_{31} -c_{12}a_{21}-c_{11}a_{11}=&0
  \\a_{13}c_{32} + a_{12}c_{22}+a_{11}c_{12}-c_{13}a_{32} - c_{12}a_{22}-c_{11}a_{12}  =&0
   \\ a_{13}c_{33} + a_{12}c_{23}+a_{11}c_{13}- c_{13}a_{33} - c_{12}a_{23}-c_{11}a_{13}=&0
    \\a_{23}c_{31} + a_{22}c_{21}+a_{21}c_{11}-  c_{23}a_{31} - c_{22}a_{21}-c_{21}a_{11} =&0
     \\a_{23}c_{32} + a_{22}c_{22}+a_{21}c_{12}-c_{23}a_{32} - c_{22}a_{22}-c_{21}a_{12} =&0
      \\ a_{23}c_{33} + a_{22}c_{23}+a_{21}c_{13}- c_{23}a_{33} - c_{22}a_{23}-c_{21}a_{13} =&0
       \\ a_{33}c_{31} + a_{32}c_{21}+a_{31}c_{11}- c_{33}a_{31} - c_{32}a_{21}-c_{31}a_{11}=&0
        \\a_{33}c_{32} + a_{32}c_{22}+a_{31}c_{12}-c_{33}a_{32} - c_{32}a_{22}-c_{31}a_{12}=&0
        \\ a_{33}c_{33} + a_{32}c_{23}+a_{31}c_{13}-c_{33}a_{33} -c_{32}a_{23}-c_{31}a_{13} =&0.
 \end{align*}
These equations can be expressed as $M_A\textbf{c}=\textbf{0}$, where
 \[M_A =\left( \begin{smallmatrix}
 0 & -a_{21} & -a_{31} & a_{12} & 0 & 0 & a_{13} & 0 & 0\\
-a_{12} & a_{11}-a_{22} & -a_{32} & 0 & a_{12} & 0 & 0 & a_{13} & 0\\ 
-a_{13} & -a_{23} & a_{11}-a_{33} & 0 & 0 & a_{12} & 0 & 0 & a_{13}\\
 a_{21} & 0 & 0 & a_{22}-a_{11} & -a_{21} & -a_{31} & a_{23} & 0 & 0\\
0 & a_{21} & 0 & -a_{12} & 0 & -a_{32} & 0 & a_{23} & 0\\ 
0 & 0 & a_{21} & -a_{13} & -a_{23} & a_{22}-a_{33} & 0 & 0 & a_{23}\\
 a_{31} & 0 & 0 & a_{32} & 0 & 0 & a_{33}-a_{11} & -a_{21} & -a_{31}\\
0 & a_{31} & 0 & 0 & a_{32} & 0 & -a_{12} & a_{33}-a_{22} & -a_{32}\\ 
0 & 0 & a_{31} & 0 & 0 & a_{32} & -a_{13} & -a_{23} & 0\\
\end{smallmatrix}\right)
 \] and $\textbf{c}=(c_{11},c_{12},c_{13},c_{21},c_{22},c_{23},c_{31},c_{32},c_{33})^T$; we remark that $M_A=A\otimes I- I\otimes A^T$, where $\otimes$ denotes the Kronecker product of matrices \cite[Chapter 4]{hj2}.
 Similarly, a matrix $M_B$ encodes the equation $BC-CB=0$.  The matrix $M_{A,B}$ is obtained by stacking $M_A$ on top of $M_B$.  So, in the $n=3$ case our equation $M_{A,B}\textbf{c}=\textbf{0}$ is
\[\left(\begin{smallmatrix}
 0 & -a_{21} & -a_{31} & a_{12} & 0 & 0 & a_{13} & 0 & 0\\
-a_{12} & a_{11}-a_{22} & -a_{32} & 0 & a_{12} & 0 & 0 & a_{13} & 0\\ 
-a_{13} & -a_{23} & a_{11}-a_{33} & 0 & 0 & a_{12} & 0 & 0 & a_{13}\\
 a_{21} & 0 & 0 & a_{22}-a_{11} & -a_{21} & -a_{31} & a_{23} & 0 & 0\\
0 & a_{21} & 0 & -a_{12} & 0 & -a_{32} & 0 & a_{23} & 0\\ 
0 & 0 & a_{21} & -a_{13} & -a_{23} & a_{22}-a_{33} & 0 & 0 & a_{23}\\
 a_{31} & 0 & 0 & a_{32} & 0 & 0 & a_{33}-a_{11} & -a_{21} & -a_{31}\\
0 & a_{31} & 0 & 0 & a_{32} & 0 & -a_{12} & a_{33}-a_{22} & -a_{32}\\ 
0 & 0 & a_{31} & 0 & 0 & a_{32} & -a_{13} & -a_{23} & 0\\
0 & -b_{21} & -b_{31} & b_{12} & 0 & 0 & b_{13} & 0 & 0\\
-b_{12} & b_{11}-b_{22} & -b_{32} & 0 & b_{12} & 0 & 0 & b_{13} & 0\\ 
-b_{13} & -b_{23} & b_{11}-b_{33} & 0 & 0 & b_{12} & 0 & 0 & b_{13}\\
 b_{21} & 0 & 0 & b_{22}-b_{11} & -b_{21} & -b_{31} & b_{23} & 0 & 0\\
0 & b_{21} & 0 & -b_{12} & 0 & -b_{32} & 0 & b_{23} & 0\\ 
0 & 0 & b_{21} & -b_{13} & -b_{23} & b_{22}-b_{33} & 0 & 0 & b_{23}\\
 b_{31} & 0 & 0 & b_{32} & 0 & 0 & b_{33}-b_{11} & -b_{21} & -b_{31}\\
0 & b_{31} & 0 & 0 & b_{32} & 0 & -b_{12} & b_{33}-b_{22} & -b_{32}\\ 
0 & 0 & b_{31} & 0 & 0 & b_{32} & -b_{13} & -b_{23} & 0\\
\end{smallmatrix}\right)
\left(\begin{smallmatrix}
c_{11}\\c_{12}\\c_{13}\\c_{21}\\c_{22}\\c_{23}\\c_{31}\\c_{32}\\c_{33}\\
\end{smallmatrix}\right)=
\left(\begin{smallmatrix}
0\\0\\0\\0\\0\\0\\0\\0\\0\\0\\0\\0\\0\\0\\0\\0\\0\\0\\
\end{smallmatrix}\right).\]
To find the defining equations for $\mathcal{C}_3^2$ in the polynomial ring $k[x_{11},\ldots,x_{33},y_{11},\ldots,y_{33}]$, we would consider the matrix
\[\left(\begin{smallmatrix}
 0 & -x_{21} & -x_{31} & x_{12} & 0 & 0 & x_{13} & 0 & 0\\
-x_{12} & x_{11}-x_{22} & -x_{32} & 0 & x_{12} & 0 & 0 & x_{13} & 0\\ 
-x_{13} & -x_{23} & x_{11}-x_{33} & 0 & 0 & x_{12} & 0 & 0 & x_{13}\\
 x_{21} & 0 & 0 & x_{22}-x_{11} & -x_{21} & -x_{31} & x_{23} & 0 & 0\\
0 & x_{21} & 0 & -x_{12} & 0 & -x_{32} & 0 & x_{23} & 0\\ 
0 & 0 & x_{21} & -x_{13} & -x_{23} & x_{22}-x_{33} & 0 & 0 & x_{23}\\
 x_{31} & 0 & 0 & x_{32} & 0 & 0 & x_{33}-x_{11} & -x_{21} & -x_{31}\\
0 & x_{31} & 0 & 0 & x_{32} & 0 & -x_{12} & x_{33}-x_{22} & -x_{32}\\ 
0 & 0 & x_{31} & 0 & 0 & x_{32} & -x_{13} & -x_{23} & 0\\
0 & -y_{21} & -y_{31} & y_{12} & 0 & 0 & y_{13} & 0 & 0\\
-y_{12} & y_{11}-y_{22} & -y_{32} & 0 & y_{12} & 0 & 0 & y_{13} & 0\\ 
-y_{13} & -y_{23} & y_{11}-y_{33} & 0 & 0 & y_{12} & 0 & 0 & y_{13}\\
 y_{21} & 0 & 0 & y_{22}-y_{11} & -y_{21} & -y_{31} & y_{23} & 0 & 0\\
0 & y_{21} & 0 & -y_{12} & 0 & -y_{32} & 0 & y_{23} & 0\\ 
0 & 0 & y_{21} & -y_{13} & -y_{23} & y_{22}-y_{33} & 0 & 0 & y_{23}\\
 y_{31} & 0 & 0 & y_{32} & 0 & 0 & y_{33}-y_{11} & -y_{21} & -y_{31}\\
0 & y_{31} & 0 & 0 & y_{32} & 0 & -y_{12} & y_{33}-y_{22} & -y_{32}\\ 
0 & 0 & y_{31} & 0 & 0 & y_{32} & -y_{13} & -y_{23} & 0\\
\end{smallmatrix}
\right)
\]  and use as the defining equations all determinants of $8\times 8$ submatrices. This yields ${9\choose 8}\cdot{18\choose 8}=393\,822$ polynomial equations. 
\end{example}

\begin{remark}
Specializing to the cases of  $k=\mathbb{R}$ or $k=\mathbb{C}$, the fields of real or complex numbers, Proposition \ref{prop:minors} leads to a nice property:  since affine varieties are closed in the Euclidean topology, the set $\mathcal{C}^{2}_n$ is closed.  For instance, if $\{A_i\}_{i=1}^\infty\subset\mathbb{C}^{n\times n}$ and $\{B_i\}_{i=1}^\infty\subset\mathbb{C}^{n\times n}$ are sequences of matrices such that
\begin{itemize}
\item[(i)] $\lim_{i\rightarrow\infty}A_i=A$ and $\lim_{i\rightarrow\infty}B_i=B$, and
\item[(ii)] for each $i$, there exists a non-scalar matrix $C_i$ such that $A_i\leftrightarrow C_i\leftrightarrow B_i$,
\end{itemize} then there must exist a non-scalar matrix $C\in \mathbb{C}^{n\times n}$ such that $A\leftrightarrow C\leftrightarrow B$.  (The same holds replacing $\mathbb{C}$ with $\mathbb{R}$.)  This is not immediately obvious even with the additional assumption that the limit $\lim_{i\rightarrow \infty} C_i$ exists, since the limit of non-scalar matrices could be a scalar matrix.
\end{remark}

We now provide a direct algebro-geometric construction for $\mathcal{C}^{2}_n$ over the field $\mathbb{C}$.

\begin{proposition}
Inside of the $3n^2$-dimensional space
$\text{Mat}_{n\times n}^3$, let
$$\mathcal{T}=\{(A,C,B)\,|\, A\leftrightarrow C\leftrightarrow B\}.$$
Let $\mathcal{S}$ be the set of all triples $(A,C,B)$ where $C$ is a scalar matrix.  Finally, let 
$$\pi:\text{Mat}_{n\times n}^3\rightarrow \text{Mat}_{n\times n}^2$$
be projection onto the $A$ and $B$ coordinates, so that $\pi(A,C,B)=(A,B)$.
Then  both $\mathcal{T}$ and $\mathcal{S}$ are  affine varieties in $3n^2$-dimensional space, and
$$\mathcal{C}^{2}_n=\pi(\mathcal{T}\setminus \mathcal{S}).$$
\end{proposition}
\begin{proof}
Let $X$, $Z$, and $Y$ be matrices of variables corresponding to the coordinates of $A$, $C$, and $B$, respectively. Then $\mathcal{T}=\textbf{V}(I)$, where the ideal $I$ is generated by the polynomials $(XZ-ZX)_{ij}$ and $(YZ-ZY)_{ij}$; and  $\mathcal{S}=\textbf{V}(J)$, where the ideal $J$ is generated by the polynomials $z_{11}-z_{ii}$ (for $2\leq i\leq n$) and $z_{ij}$ (for $i\neq j$).  The fact that  $\mathcal{C}^{2}_n=\pi(\mathcal{T}\setminus \mathcal{S})$ follows by definition.
\end{proof}
Now we wish to take the set-theoretic difference of one variety from another, and then project the resulting set to a lower dimensional space.  It is possible to do this by manipulating ideals, at least up to Zariski closure:  By Propositions \ref{clo:projection} and \ref{clo:difference}, the set-minus operation corresponds to ideal saturation, and projection corresponds to elimination of variables.  In particular,
$$\textbf{V}(I:J^\infty)=\overline{\mathcal{T}\setminus \mathcal{S}},$$
and 
$$\textbf{V}((I:J^\infty)\cap k[X,Y])=\overline{\pi\left(\overline{\mathcal{T}\setminus \mathcal{S}}\right)}.$$
We will show that we may remove the Zariski closures, so that this ideal does in fact define~$\mathcal{C}_n^2$.

\begin{proposition} \label{proposition:nozariski} When $k=\mathbb{C}$, we have
$$\textbf{V}((I:J^\infty)\cap k[X,Y]))=\pi{(\mathcal{T}\setminus \mathcal{S})}.$$
\end{proposition}
\begin{proof}  Let $\mathcal{R}=\mathcal{T}\setminus \mathcal{S}$.  Then we wish to show
$$\pi(\mathcal{R})=\overline{\pi(\overline{\mathcal{R}})}. $$
Certainly we have $\pi(\mathcal{R})\subset \overline{\pi(\overline{\mathcal{R}})}$.  Since $\pi(\mathcal{R})=\mathcal{C}^2_n$ is a variety by Proposition \ref{prop:minors}, it is Zariski closed.  Thus if we can show that $\pi(\overline{\mathcal{R}})\subset \pi(\mathcal{R})$, it will follow that  $\overline{\pi(\overline{\mathcal{R}})}\subset \pi(\mathcal{R})$, since  $\overline{\pi(\overline{\mathcal{R}})}$ is the smallest Zariski-closed set containing $\pi(\overline{\mathcal{R}})$.

Suppose that $(A,B)\in {\pi(\overline{\mathcal{R}})}$.  It follows that there exists $C$ such that $(A,C,B)\in\overline{\mathcal{R}}$.  Since $\mathcal{R}$ is a constructible set, its Zariski closure is equal to its closure in the Euclidean topology by Proposition \ref{proposition:closure}.  It follows that there exists a sequence of points $\{(A_i,C_i,B_i)\}_{i=0}^\infty\subset \mathcal{R}\subset\mathcal{T}$ such that $\lim_{i\rightarrow\infty}(A_i,C_i,B_i)=(A,C,B)$.  {Then} 
$\lim_{i\rightarrow\infty}(A_i,B_i)=(A,B)$.  Since $(A_i,B_i)\in\pi(\mathcal{R})=\mathcal{C}^2_n$, and since  $\mathcal{C}^2_n$ is closed in the Euclidean topology, we have that $(A,B)=\lim_{i\rightarrow\infty}(A_i,B_i)\in \mathcal{C}^2_n$.  So, ${\pi(\overline{\mathcal{R}})}\subset \mathcal{C}^2_n$.  This completes the proof.
\end{proof}


 If the field $k$ is algebraically closed, we now determine the decomposition of $\mathcal{C}_n^2$ into irreducible varieties, as well as its dimension.  The key players will be the varieties $\overline{\mathcal{Z}_i}$ for $1\leq i\leq \lfloor\frac{n}{2}\rfloor$, where $\mathcal{Z}_i$ is the set of all pairs of matrices that commute with a common rank $i$ matrix that squares to itself:
\[\mathcal{Z}_i=\{(A,B)\,|\, \text{$\exists P\in\textrm{Mat}_{n\times n}$ s.t. $P^2=P$, $\textrm{rank}(P)=i$, $AP=PA$, $BP=PB$}\}.\]
We start with the following two lemmas.

\begin{lemma}\label{lemma:c2n_decomposition}
When $k$ is algebraically closed, we have
\[\mathcal{C}_n^2=\overline{\mathcal{Z}_1}\cup\cdots\cup \overline{\mathcal{Z}_{\lfloor n/2\rfloor}}.\]\end{lemma}

\begin{proof}
By definition $\mathcal{Z}_i\subset\mathcal{C}_n^2 $ for all $i$, and since $\mathcal{C}_n^2$ is a variety we have  $\overline{\mathcal{Z}_i}\subset\mathcal{C}_n^2 $ as well. 
This gives one direction of containment.

For the other direction, let $(A,B)\in \mathcal{C}_n^2$, so that there exists a non-scalar matrix $C'\in\textrm{Mat}_{n\times n}$ commuting with both $A$ and $B$.  
Then there exists a maximal matrix  $C \in\textrm{Mat}_{n\times n}$ commuting with both $A$ and $B$.
Here maximality is considered in the sense of \cite{dgko}, that is the centralizer of no non-scalar matrix properly contains the centralizer of $C$. Indeed, if $C'$ is not maximal, then its centralizer is contained in a centralizer of a certain non-scalar maximal matrix $C$. Hence both $A$ and $B$ commute with $C$. Since $k$ is algebraically closed by \cite[Theorem 3.2]{dgko} 
$C$ is either a nontrivial idempotent matrix, or a nonzero matrix that squares to the zero matrix. In the first case, since $C$ and $I-C$ are maximal idempotents which commute with $A$ and $B$, we have $(A,B)\in \mathcal{Z}_i$ where $i=\min\{\textrm{rank}(C),\textrm{rank}(I-C)\}$. Now assume that $C$ is a nonzero matrix that squares to $0$. Permuting the rows and columns in the Jordan canonical form of $C$ in order to collect the zero rows together  we 
obtain that $C$ is similar to the matrix 
$C''=\begin{pmatrix}
0 & I & 0\\
0 & 0 & 0\\
0 & 0 & 0
\end{pmatrix}$ where the first two block rows and columns are of size $i$ and the last ones are of size $n-2i$ (possibly 0). Without loss of generality we  further assume that $C=C''$. Commutativity then implies that $A=\begin{pmatrix}
A_1 & A_2 & A_3\\
0   & A_1 & 0\\
0   & A_4 & A_5
\end{pmatrix}$ and
$B=\begin{pmatrix}
B_1 & B_2 & B_3\\
0   & B_1 & 0\\
0   & B_4 & B_5
\end{pmatrix}$ for some submatrices $A_j$ and $B_j$. Now we 
consider
$X=\begin{pmatrix}
0 & 0 & 0\\
0 & I & 0\\
0 & 0 & I
\end{pmatrix}$,
$Y=\begin{pmatrix}
0 & 0   & 0\\
0 & A_2 & A_3\\
0 & 0   & 0
\end{pmatrix}$ and
$Z=\begin{pmatrix}
0 & 0   & 0\\
0 & B_2 & B_3\\
0 & 0   & 0
\end{pmatrix}$. Then $C+\lambda X$ commutes with $A+\lambda Y$ and with $B+\lambda Z$ for each $\lambda \in k$, so $(A+\lambda Y,B+\lambda Z)\in \mathcal{C}_n^2$ for each $\lambda \in k$. Moreover, for $\lambda \ne 0$ the matrix $C+\lambda X$ is a multiple of an idempotent of rank $n-i$, therefore $(A+\lambda Y,B+\lambda Z)\in \mathcal{Z}_i$ for each $\lambda \ne 0$, and consequently $(A,B)\in \overline{\mathcal{Z}_i}$.
\end{proof}

\begin{lemma}\label{lemma:z_dimensions}
For $1\leq i\leq \lfloor\frac{n}{2}\rfloor$, the variety $\overline{\mathcal{Z}_i}$ is irreducible, with dimension $n^2+i^2+(n-i)^2$.
\end{lemma}

\begin{proof}
Fix $P$ to be an idempotent of rank $i$.  The set $\mathcal{Z}_P=\{(A,B)\,|\, AP=PA, BP=PB\}\subset\mathcal{C}_n^2$ is an affine space of dimension $2(i^2+(n-i)^2)$, and is therefore irreducible.  Now, the group $\textrm{GL}_n$ acts on $\textrm{Mat}_{n\times n}$ by conjugation, which induces the $\textrm{GL}_n$-action on $\textrm{Mat}_{n\times n}^2$ by simultaneous conjugation: $(g,(A,B))\mapsto (gAg^{-1},gBg^{-1})$. Let $\cdot$ denote this action. 
Consider the orbit of $\mathcal{Z}_P$ under this action.  Since any two idempotent matrices of the same rank are similar, we have that $\mathcal{Z}_i=\textrm{GL}_n\cdot \mathcal{Z}_P$.  Since $\textrm{GL}_n$ is a connected group acting on an irreducible variety $\mathcal{Z}_P$, we have that $\overline{\mathcal{Z}_i}$ is irreducible.  (To see this, let $\varphi: \textrm{GL}_n\times \textrm{Mat}_{n\times n}^2\rightarrow \textrm{Mat}_{n\times n}^2$ denote the morphism of the group action, so that $\mathcal{Z}_i$ is the image of  $\textrm{GL}_n\times \mathcal{Z}_P$ under $\varphi$.  Since $\textrm{GL}_n$ and $\mathcal{Z}_P$ are irreducible, so too is their product, and so too is the image of the product under $\varphi$.)

We are now ready to compute the dimension of $\overline{\mathcal{Z}_i}$.  Choose $P$ to be the idempotent matrix $\left(\begin{smallmatrix}I&0\\0&0\end{smallmatrix}\right)$, where $I$ is the $i\times i$ identity matrix.  Consider the following map:
\begin{align*}
\varphi:\textrm{GL}_n\times \mathcal{Z}_P&\rightarrow \mathcal{Z}_i
\\(g,(A,B))&\mapsto g\cdot (A,B)= (gAg^{-1},g Bg^{-1}).
\end{align*}
This map is surjective since $\mathcal{Z}_i=\textrm{GL}_n\cdot \mathcal{Z}_P
$, so we can determine the dimension of $\mathcal{Z}_i$ by computing the generic dimension of a fibre of $\varphi$.  To accomplish this, we note that the set of all pairs $(A,B)\in\mathcal{C}_n^2$ where $A$ has $n$ distinct eigenvalues is an open subset of $\mathcal{C}_n^2$; this open set intersects every $\mathcal{Z}_i$.  Assume that $\varphi(g,(A,B))=\varphi(g',(A',B'))$, where $A$ has $n$ distinct eigenvalues.  Conjugating our four red block diagonal matrices $A$, $B$, $A'$, $B'$ with an element of $\textrm{GL}_i\times \textrm{GL}_{n-i}$, we may assume without loss of generality that $A$ is a diagonal matrix, since it has $n$ distinct eigenvalues.  Since $A'$ is similar to $A$, it is also diagonalizable. Since it must already be block diagonal, we know that there exists $h\in \textrm{GL}_i\times \textrm{GL}_{n-i}$ such that $hA'h^{-1}$ is diagonal.  Now, since $A$ and $hA'h^{-1}$ are similar diagonal matrices, there exists a permutation matrix $\sigma\in\textrm{GL}_n$ with $\sigma hA'h^{-1}\sigma^{-1}=A$.  Now, the equality $\varphi(g,(A,B))=\varphi(g',(A',B'))$ is equivalent to $A'=g'^{-1}gAg^{-1}g'$ and  $B'=g'^{-1}gBg^{-1}g'$. Thus, when $g'$ is chosen, we have that $A'$ and $B'$ are uniquely determined by $A,B$, and $g$.  Note that $g'$ has to satisfy the equation $\sigma h g'^{-1}g  Ag^{-1}g'h^{-1}\sigma^{-1}=A$, i.e. that $A$ commutes with $g^{-1}g'h^{-1}\sigma^{-1}$.  As $A$ is diagonal with $n$ distinct eigenvalues, we have that $g'=g D\sigma h$ for some invertible diagonal matrix $D$.  Since the permutation matrices normalize the set of diagonal matrices, we have that $g'=g\sigma D' h$ for some invertible diagonal matrix $D'$.  Letting $\mathfrak{S}_n$ denote the group of permutation matrices, we have that $g'\in g\mathfrak{S}_n(\textrm{GL}_i\times \textrm{GL}_{n-i})$.  Thus the fibre $\varphi^{(-1)}(\varphi(g,(A,B)))$ is a subvariety of the variety that is parametrized by the product $\mathfrak{S}_n(\textrm{GL}_i\times \textrm{GL}_{n-i})$.  Since $\mathfrak{S}_n$ is finite, the dimension of each irreducible component of the fibre is at most $\dim(\textrm{GL}_i\times \textrm{GL}_{n-i})=i^2+(n-i)^2$, giving us an upper bound on dimension.  For a lower bound, consider the set
\[\{(gh,(h^{-1}Ah,h^{-1}Bh))\,|\, h\in \textrm{GL}_i\times \textrm{GL}_{n-i}\}\]
which contains $(g,(A,B))$.
This set has dimension $i^2+(n-i)^2$, and is a subset of the fibre $\varphi^{(-1)}(\varphi(g,(A,B)))$.  Thus when $A$ has $n$ distinct eigenvalues, we have that the dimension of the fibre over $(g,(A,B))$ is $i^2+(n-i)^2$.  As this result holds on the open dense set where $A$ has $n$ distinct eigenvalues, we have that $\textrm{dim}(\mathcal{Z}_i)=n^2+i^2+(n-i)^2$ by \cite[Theorem 11.12]{har}.
\end{proof}

It immediately follows that $\dim(\mathcal{C}_n^2)=n^2+1^2+(n-1)^2=2n^2-2n+2$, as this is the largest dimension of any $\mathcal{Z}_i$.  These lemmas also allow us to characterize the decomposition of $\mathcal{C}_n^2$ into irreducibles.

\begin{theorem}\label{theorem:decomposition}
Let $k$ be an algebraically closed field.
 The unique irredundant decomposition of $\mathcal{C}_n^2$ into irreducible varieties is
 \[\mathcal{C}_n^2=\overline{\mathcal{Z}_1}\cup\cdots\cup \overline{\mathcal{Z}_{\lfloor n/2\rfloor}}.\]
\end{theorem}

\begin{proof}

By Lemmas \ref{lemma:c2n_decomposition} and \ref{lemma:z_dimensions}, all we need to show is that no $\overline{\mathcal{Z}_i}$ is contained in any $\overline{\mathcal{Z}_j}$ for $i\neq j$. For each $i=1,\ldots,\lfloor\frac{n}{2}\rfloor$, define
\begin{align*}
\mathcal{W}_i\,=\,& \left\{  (A,B)\in\textrm{Mat}_{n\times n}^2\,|\, \exists \textrm{ subspaces }U,V\subset k^n \textrm{ s.t. } \dim(U)=i,\dim(V)=n-i,\right.
\\& \left.AU\subset U, AV \subset V, BU\subset U, BV\subset V \right\}.
\end{align*}
We have that $\mathcal{Z}_i$ is contained in $\mathcal{W}_i$ for each $i$.  Moreover, the set $\mathcal{W}_i$ is the image of the projection of the set
\begin{align*}
\mathcal{W}_i'\,=\,& \left\{  (A,B,U,V)\in\textrm{Mat}_{n\times n}\times \textrm{Mat}_{n\times n}\times\textrm{Gr}(n,i)\times \textrm{Gr}(n,n-i) \,|\, \right.
\\& \left.AU\subset U, AV \subset V, BU\subset U, BV\subset V \right\}
\end{align*}
to the first two factors, where $\textrm{Gr}$ denotes the Grassmannian. Recall that the Grassmannian $\textrm{Gr}(n,j)$ is the set of all $j$-dimensional subspaces of an $n$-dimensional vector space. It has a natural structure of a projective variety given by Pl\"ucker embedding $\textrm{Gr}(n,j)\to \mathbb{P}(\wedge^j(k^n))=\mathbb{P}^{{n\choose j}-1}$ that sends a subspace $U\in \textrm{Gr}(n,j)$ with a basis $\{u_1,\ldots ,u_j\}$ to $u_1\wedge \cdots \wedge u_j\in \mathbb{P}(\wedge^j(k^n))$, see \cite[Chapter 6]{har}. Note that this map is well-defined, as the wedge-product $u_1\wedge \cdots \wedge u_j$ is independent of the chosen basis.  If $U\in \textrm{Gr}(n,i)$ is represented in Pl\"{u}cker coordinates by $u_1\wedge \ldots\wedge u_i$, then $AU\subset U$ if and only if $(A+sI)u_1\wedge\ldots \wedge (A+sI)u_i$ is a multiple of $u_1\wedge \ldots\wedge u_i$ for each $s\in k$ \cite[Theorem 2.2]{tsat}, which is equivalent to vanishing of all $2\times 2$ minors of the $2\times {n\choose i}$ matrix
$\begin{pmatrix}
(A+sI)u_1\wedge\ldots \wedge (A+sI)u_i\\
u_1\wedge \ldots\wedge u_i
\end{pmatrix}$ for all $s\in k$.  This is a polynomial condition which is homogeneous in the $u$-coordinates, so $\mathcal{W}_i'$ is a  closed subset of $k^{2n^2}\times \mathbb{P}^{{n\choose i}-1}\times \mathbb{P}^{{n\choose n-i}-1}$. By Proposition \ref{proposition:proper} its projection onto the first two factors, which is $\mathcal{W}_i$, is closed, and $\overline{\mathcal{Z}_i}\subset \mathcal{W}_i$.  Now, for $1\leq i<j\leq \lfloor \frac{n}{2}\rfloor$ there exists a pair of matrices $(A,B)\in \mathcal{Z}_j$ that do not have an $(n-i)$-dimensional common invariant subspace, implying $\mathcal{Z}_j\not\subset \overline{\mathcal{Z}_i}$.  On the other hand, $\mathcal{Z}_i$ cannot be a subset of $\overline{\mathcal{Z}_j}$ since $\dim\mathcal{Z}_i>\dim \mathcal{Z}_j$.  This completes the proof.
\end{proof}

One consequence of this theorem is that $\mathcal{C}^2_n$ is an irreducible variety if and only if $n\leq 3$:  for $n=2$ we have $\mathcal{C}^2_2=\mathcal{C}^1_2$ and for $n=3$ we have $\mathcal{C}^2_3=\overline{\mathcal{Z}_1}$; but for $n\geq 4$ the variety $\mathcal{C}^2_n$ has at least two distinct irreducible components, namely $\overline{\mathcal{Z}_1}$ and $\overline{\mathcal{Z}_2}$.

\section{The distance-$3$ commuting set}
\label{section:distance3}

To study pairs of matrices with commuting distance at most $3$, we first define the notion of polynomially commuting matrices.
  
\begin{definition}
Fix $n\geq 3$, and  let $A$ and $B$ be two matrices in $\text{Mat}_{n\times n}$. Then	
$A$ and $B$ \emph{polynomially commute} if there exist polynomials $p,q\in k[x]$ such that
\begin{itemize}
\item[(i)] $p(A)\leftrightarrow q(B)$, and
\item[(ii)] $1\leq \deg(p),\deg(q)\leq n-1$.
\end{itemize} 
\end{definition}
Note that without the degree bounds on $p$ and $q$, every pair of matrices would polynomially commute with one another:  if $\deg(p)=0$, then $p(A)$ would be a scalar matrix; and if $\deg(p)=n$ we could choose $p$ equal to the characteristic polynomial of $A$, so that $p(A)$ is the zero matrix.  In either case, we would have $p(A)$ commuting with every matrix.  However, unlike with our definition of commuting distance, it is not forbidden for $p(A)$ or $q(B)$ to be a scalar matrix.  This will come into play in Lemma \ref{lemma:derogatory}.

We also remark that we may assume that $p(x)$ and $q(x)$ have no constant term.  This is because for any constants $\lambda$ and $\mu$, two matrices $C$ and $D$ commute if and only if $C-\lambda I$ and $D-\mu I$ commute.  Eliminating the constant term from $p$ and $q$ only effects $p(A)$ and $p(B)$ by subtracting off a constant multiple of the identity matrix.

Let $\mathcal{PC}_n\subset \text{Mat}_{n\times n}^2$ denote the set of all pairs of polynomially commuting matrices.  To prove that $\mathcal{C}^{3}_n$ is an affine variety over an algebraically closed
field, we will first show that $\mathcal{C}^{3}_n=\mathcal{PC}_n$.  We will then show that $\mathcal{PC}_n$ is an affine variety, at least over an algebraically closed field.  We begin with the following result.

\begin{lemma}\label{lemma:derogatory} Let $A,B\in \text{Mat}_{n\times n}$ such that at least one of $A$ and $B$ is derogatory.  Then $(A,B)\in\mathcal{PC}_n$.
\end{lemma}
\begin{proof}  Assume that $A$ is derogatory; a symmetric argument will hold when $B$ is derogatory. Let $p\in k[x]$ be the minimal polynomial of $A$. Then by the definition of derogatory matrices, $p(x)$ is not equal to the characteristic polynomial of $A$, which has degree $n$.  Since the minimal polynomial divides the characteristic polynomial, we have $\deg(p)<n$.  Combined with the fact that minimal polynomial of any matrix is non-constant, this implies that $1\leq \deg(p)\leq n-1$.  Now let $q(x)=x$.  We then have
$$A\leftrightarrow p(A)\leftrightarrow q(B)\leftrightarrow B,$$
since the zero matrix $p(A)$ commutes with all other matrices and since $B$ commutes with itself.  Thus $A$ and $B$ polynomially commute.
\end{proof}

\begin{proposition}\label{prop:distance3pc} Suppose $k$ is algebraically closed.  Then the distance-$3$ commuting set is equal to the set of polynomially commuting pairs of matrices:
$$\mathcal{C}^{3}_n=\mathcal{PC}_n.$$
\end{proposition}

\begin{proof}  Let $(A,B)\in \text{Mat}_{n\times n}^2$.  We claim that $(A,B)\in\mathcal{C}^{3}_n$ if and only if $(A,B)\in \mathcal{PC}_n$.

First suppose at least one of $A$ and $B$ is derogatory.  Then by 
Theorem \ref{theorem:equivalences}, 
we have $(A,B)\in\mathcal{C}^{3}_n$,  and by Lemma \ref{lemma:derogatory}, we have $(A,B)\in \mathcal{PC}_n$.  Thus our claim holds in this case.

Now suppose $A$ and $B$ are both non-derogatory.  By Theorem \ref{theorem:equivalences}, $A$   commutes with polynomials in $A$ only,  and $B$  commutes with polynomials in $B$ only.  Now, $(A,B)\in \mathcal{C}^{3}_n$ if and only if there exist non-scalar matrices $C$ and $D$ such that
$$A\leftrightarrow C\leftrightarrow D\leftrightarrow B. $$
Since $A$ and $B$  commute with polynomials in themselves only, such $C$ and $D$ exist if and only if there are polynomials $p(x),q(x)\in k[x]$ such that $C=p(A)$ and $D=q(B)$. In fact, $p$ and $q$ must have degree at least $1$, since $C$ and $D$ are non-scalar.  Moreover, $p$ may be chosen to have degree at most $n-1$, since the characteristic polynomial is annihilating, so any power $A^k$ with $k\geq n$ 
can be written as a combination of $I, A, A^2,\ldots,A^{n-1}$. A similar argument shows  we may take $\deg(q)\leq n-1$.  Thus, the desired matrices $C$ and $D$ exist if and only if $(A,B)\in \mathcal{PC}_n$.  This completes our proof.
\end{proof}

To show that $\mathcal{C}_n^{3}$ is an affine variety in the case of algebraically closed field, it remains to show the following proposition.

\begin{proposition}\label{prop:pcaffine}  Suppose $k$ is algebraically closed.  Then the set $\mathcal{PC}_n$ is an affine variety.
\end{proposition}

\begin{proof} Suppose $A$ and $B$ polynomially commute, and let $p$ and $q$ be polynomials satisfying all the conditions of our definition.  As previously noted, we may assume that $p$ and $q$ have no constant term.  Write
$$p(x)=\sum_{i=1}^{n-1}c_ix^i,\,\,\,\,q(x)=\sum_{i=1}^{n-1}d_ix^i.$$
The condition $p(A)q(B)-q(B)p(A)=0$ can then be written as
\begin{align*}
0=&\left(\sum_{i=1}^{n-1}c_iA^i\right)\left(\sum_{i=1}^{n-1}d_iB^i\right)-\left(\sum_{i=1}^{n-1}d_iB^i\right)\left(\sum_{i=1}^{n-1}c_iA^i\right)
\\=&\left(\sum_{i=1}^{n-1}\sum_{j=1}^{n-1}c_id_jA^iB^j\right) -\left(\sum_{i=1}^{n-1}\sum_{j=1}^{n-1}c_id_jB^jA^i\right)
\\=&\sum_{i=1}^{n-1}\sum_{j=1}^{n-1} c_id_j\left(A^iB^j-B^jA^i\right).
\end{align*}
Thinking of the $c_i$, the $d_j$, the entries of $A$, and the entries of $B$ as variables, this equation really consists of $n^2$ polynomials (one for each coordinate) in those variables. These equations are bilinear in the $c_i$ and the $d_j$ variables.

Note that the set of polynomials with no constant term and degree between $1$ and $n-1$ is a vector space.  Call this vector space $V$, and let $\mathbb{P}(V)$ be the projectivization of $V$, so that two polynomials are identified if and only if they differ by a constant multiple. The equivalence class of a polynomial $p\in V$ in $\mathbb{P}(V)$ will be denoted by $\langle p\rangle$. Consider the incidence correspondence
$$\Phi = \{(A,B,\langle p\rangle,\langle q\rangle\,|\, p(A)q(B)=q(B)p(A)\}\subset \text{Mat}_{n\times n}^2\times \mathbb{P}(V)\times\mathbb{P}(V).$$
Note that $\Phi$ is Zariski closed: it is defined by the equations we derived from $p(A)q(B)-q(B)p(A)=0$, which are indeed bilinear in the coordinates defining $p$ and $q$.  Consider the projection map 
$$\pi:\text{Mat}_{n\times n}^2\times \mathbb{P}(V)\times\mathbb{P}(V)\rightarrow \text{Mat}_{n\times n}^2.$$
By Proposition \ref{proposition:proper}, $\pi(Z)\subset \text{Mat}_{n\times n}^2$ is Zariski-closed whenever whenever $Z$ is.  Thus, the image $\pi(\Phi)=\mathcal{PC}_n$ is Zariski-closed in $\text{Mat}_{n\times n}^2$.  We conclude that $\mathcal{PC}_n$ is an affine variety.
\end{proof}

Combining Propositions \ref{prop:distance3pc} and \ref{prop:pcaffine}, we have the following result, which completes the proof of Theorem \ref{theorem:main}.

\begin{proposition}\label{theorem:distance3} Over an algebraically closed field, the distance-$3$ commuting set $\mathcal{C}_n^3$ is an affine variety.
\end{proposition}

\begin{remark}
We briefly mention another method of proof for Proposition \ref{theorem:distance3} with the additional assumption that the characteristic of $k$ does not divide $n$, suggested to the authors by Eyal Markman.  Consider all quadruples $(A,B,C,D)$ with $A\leftrightarrow B\leftrightarrow C\leftrightarrow D$, where $B$ and $C$ are nonscalar.  Without loss of generality, we may assume that $B$ and $C$ have trace zero, since adding a scalar matrix to each does not affect our assumptions.  The set of all trace zero nonscalar matrices can be identified with an $n^2-2$ dimensional projective space, and so we may think of our set of quadruples as a Zariski-closed subset of $\textrm{Mat}_{n\times n}\times \mathbb{P}^{n^2-2}\times \mathbb{P}^{n^2-2}\times \textrm{Mat}_{n\times n}$.  Projecting away from the projective spaces yields $\mathcal{C}_n^3$, which must therefore be Zariski-closed; here we use the fact that no nonzero scalar matrix has trace $0$ since the characteristic of $k$ does not divide $n$.
\end{remark}

The fact that the distance-$d$ commuting sets over algebraically closed fields are all varieties allows us to deduce the following fact.

\begin{corollary}\label{theorem:measurezero} If $k=\mathbb{C}$ or $k=\mathbb{R}$, then
$$\{(A,B)\,|\, d(A,B)\leq 3\}\subset\text{Mat}_{n\times n}(k)^2$$ is a set of
measure $0$.
\end{corollary}
 Phrased contrapositively, a ``random'' pair of matrices $(A,B)$  satisfies $d(A,B)=4$.
\begin{proof}
For $k=\mathbb{C}$, this follows from the fact that any proper subvariety of affine space over $\mathbb{C}$ or $\mathbb{R}$ has measure $0$, and the fact that at least one pair of matrices has $d(A,B)=4$. 

For $k=\mathbb{R}$, we note that $\mathcal{C}_n^3(\mathbb{R})\subset \mathcal{C}_n^3(\mathbb{C}) \cap\textrm{Mat}_{n\times n}\left(\mathbb{R}\right)^2$: this is because if there is a commuting chain over $\mathbb{R}$ connecting $A$ and $B$, the same commuting chain exists over $\mathbb{C}$.  By Lemma \ref{lemma:realclosed}, we know that $\mathcal{C}_n^3(\mathbb{C}) \cap\textrm{Mat}_{n\times n}\left(\mathbb{R}\right)^2$ is a variety since $\mathbb{C}$ is algebraic over $\mathbb{R}$; it is in fact a proper variety by the usual examples of distance-$4$ pairs of matrices.  Thus $\mathcal{C}_n^3(\mathbb{C}) \cap\textrm{Mat}_{n\times n}\left(\mathbb{R}\right)^2$ has measure $0$ in $\textrm{Mat}_{n\times n}\left(\mathbb{R}\right)^2$, and so too does its subset $\mathcal{C}_n^3(\mathbb{R})$.
\end{proof}

Just as with $\mathcal{C}_n^2$, we have an algebro-geometric construction of $\mathcal{C}^3_n$ over $\mathbb{C}$.  
\begin{proposition} Let $k={\mathbb C}$, $I$ be the ideal generated by the $3n^2$ polynomials $(XZ-ZX)_{ij}$, $(ZW-WZ)_{ij}$, and $(YW-WY)_{ij}$, $J_1$ be  the ideal generated by the polynomials $z_{11}-z_{ii}$ (for $2\leq i\leq n$) and $z_{ij}$ (for $i\neq j$), and let $J_2$ be generated by the polynomials  $w_{11}-w_{ii}$ (for $2\leq i\leq n$) and $w_{ij}$ (for $i\neq j$). Let $J=J_1\cap J_2$ be the intersection of the two ideals. Then 
	$$\mathcal{C}_n^3=\textbf{V}((I:J^\infty)\cap k[X,Y]).$$
\end{proposition}
\begin{proof}
Inside of the $4n^2$-dimensional space
$\text{Mat}_{n\times n}^4$, we consider the set
$$\mathcal{Q}=\{(A,C,D,B)\,|\, A\leftrightarrow C\leftrightarrow D\leftrightarrow B\}.$$
Let $\mathcal{S}$ be the set of all quadruples $(A,C,D,B)$ where at least one of $C$ and $D$ is a scalar matrix.  Finally, let 
$$\pi:\text{Mat}_{n\times n}^4\rightarrow \text{Mat}_{n\times n}^2$$
be projection onto the $A$ and $B$ coordinates, so that $\pi(A,C,D,B)=(A,B)$.  Then, by definition,
$$\mathcal{C}^{3}_n=\pi(\mathcal{Q}\setminus \mathcal{S}).$$
As in the $d=2$ case, we can show that $\mathcal{Q}$ and $\mathcal{S}$ are affine varieties in $4n^2$-dimensional space.  Let $X$, $Z$, $W$, and $Y$ be matrices of variables corresponding to the coordinates of $A$, $C$, $D$, and $B$, respectively. Then $\mathcal{Q}=\textbf{V}(I)$. By conditions, $\textbf{V}(J_1)$ is the set of all quadruples $(A,C,D,B)$ where $C$ is a scalar matrix, and $\textbf{V}(J_2)$ is the set of all quadruples $(A,C,D,B)$ where $D$ is a scalar matrix.  Let $J=J_1\cap J_2$ be the intersection of the two ideals.  Then $\textbf{V}(J)=\textbf{V}(J_1)\cup\textbf{V}(J_2)$ \cite[Theorem 4.3.15]{CLO}.  Thus, $\mathcal{S}=\textbf{V}(J_1)\cup \textbf{V}(J)=\textbf{V}(J)$.  Using Propositions \ref{clo:projection} and \ref{clo:difference}, we have
$$\textbf{V}(I:J^\infty)=\overline{\mathcal{Q}\setminus \mathcal{S}}$$
and 
$$\textbf{V}((I:J^\infty)\cap k[X,Y])=\overline{\pi\left(\overline{\mathcal{Q}\setminus \mathcal{S}}\right)}.$$
As argument identical to that in the proof of Proposition \ref{proposition:nozariski} shows that we may remove the Zariski closures in $\overline{\pi\left(\overline{\mathcal{Q}\setminus \mathcal{S}}\right)}$, so that
$$\textbf{V}((I:J^\infty)\cap k[X,Y])=\overline{\pi\left(\overline{\mathcal{Q}\setminus \mathcal{S}}\right)}=\pi(\mathcal{Q}\setminus\mathcal{S})=\mathcal{C}_n^3,$$
as desired.	
\end{proof}

We now study the dimension of the distance-$3$ commuting set over an algebraically closed field.  In a forthcoming paper, we will provide a computation of its dimension as being precisely \(2n^2-2\); here we content ourselves with a lower bound.

\begin{lemma}\label{lemma:derogatory_variety}
Let $V\subset \textrm{Mat}_{n\times n}$ denote the set of all derogatory matrices.  We have that $V$ is a variety of dimension $n^2-3$.
\end{lemma}

\begin{proof}
First we argue that $V$ is a variety.  Recall that a matrix $A$ is derogatory if and only if its characteristic polynomial is not equal to its minimal polynomial.  This means that there is a nontrivial linear dependence between $I,A, A^2,\cdots,A^{n-2}$.  Flattening these $n-1$ square matrices and forming a $(n-1)\times n^2$ matrix $B$, the existence of such a nontrivial linear dependence is equivalent to $B$ having rank at most $n-2$.  This rank condition is equivalent to the vanishing of all size $n-1$ minors of $B$, which are polynomials in the entries of $A$.   The set of all derogatory matrices is thus a variety defined by these polynomials.

Consider the set $V'$ of all matrices similar to $\textrm{diag}(\lambda_1,\lambda_1,\lambda_2,\ldots,\lambda_{n-1})$ for some pairwise distinct $\lambda_1,\ldots,\lambda_{n-1}\in k$.  We claim that $V'$ is dense in $V$. To see this, we can write any derogatory matrix in Jordan canonical form, and then perturb the coefficients to guarantee $n-1$ distinct eigenvalues.  Now we observe that $V'$ is the image of the map $\varphi\colon\textrm{GL}_n\times k^{n-1}\rightarrow \textrm{Mat}_{n\times n}$ that sends $(g,(\lambda_1,\ldots,\lambda_{n-1}))$ to $g\, \textrm{diag}(\lambda_1,\lambda_1,\lambda_2,\ldots,\lambda_{n-1})g^{-1}$.  If $g\, \textrm{diag}(\lambda_1,\lambda_1,\lambda_2,\ldots,\lambda_{n-1})g^{-1}=h\, \textrm{diag}(\mu_1,\mu_1,\mu_2,\ldots,\mu_{n-1})h^{-1},$
then $\mu_1=\lambda_1$ and the sequence $(\mu_2,\ldots ,\mu_{n-1})$ is a permutation of $(\lambda_2,\ldots ,\lambda_{n-1})$. As in the proof of Lemma \ref{lemma:z_dimensions} we then see that there exists a permutation matrix $\sigma\in \mathfrak{S}_{n-2}$ such that for $\rho=\begin{pmatrix}
I &0\\
0 & \sigma
\end{pmatrix}\in \mathfrak{S}_n$ the matrix $\rho^{-1}h^{-1}g$ commutes with $ \, \textrm{diag}(\lambda_1,\lambda_1,\lambda_2,\ldots,\lambda_{n-1})$ 
and is hence of the form
$\begin{pmatrix}
S&0\\
0 & D
\end{pmatrix}$ for some $2\times 2$ invertible matrix $S$ and some $(n-2)\times (n-2)$ invertible diagonal matrix $D$. Hence, $h=g
\begin{pmatrix}
S^{-1} & 0\\
0& D^{-1}
\end{pmatrix}
\begin{pmatrix}
I& 0\\
0 & \sigma^{-1}
\end{pmatrix}$, which shows that every fibre of $\varphi$ is isomorphic to $\textrm{GL}_2\times (k^*)^{n-{2}}\times \mathfrak{S}_{n-2}$, where $k^*$ denotes the nonzero elements of $k$.  This has dimension $n+2$, and so  by \cite[Theorem 11.12]{har} it follows that $\dim(V')=n^2-3$, as desired.
\end{proof}

It immediately follows from this lemma that $\dim(\mathcal{C}_n^3)\geq 2n^2-3$, since $V\times\textrm{Mat}_{n\times n}\subset \mathcal{C}_n^3$.  The following result gives us a sharper bound.

\begin{proposition}\label{c3_lower_bound}
We have $\dim(\mathcal{C}_n^3)\geq 2n^2-2$.
\end{proposition}

\begin{proof}
As previously remarked, we have $\dim(\mathcal{C}_n^3)\geq 2n^2-3$ since $V\times\textrm{Mat}_{n\times n}\subset \mathcal{C}_n^3$.  Suppose for the sake of contradiction that $\dim(\mathcal{C}_n^3)= 2n^2-3$. The dimension of a variety is by definition the largest possible dimension of an irreducible component of that variety. Therefore there exists an irreducible component $\mathcal{C}$ of $V\times \textrm{Mat}_{n\times n}$ of dimension $2n^2-3$. It is an irreducible subvariety of $\mathcal{C}_n^3$, therefore it is contained in an irreducible component $\widetilde{\mathcal{C}}$ of $\mathcal{C}_n^3$. As $\dim \mathcal{C}_n^3=2n^2-3$ by our assumption, it follows that $\dim \widetilde{\mathcal{C}}=2n^2-3$. Now, $\mathcal{C}$ and $\widetilde{\mathcal{C}}$ are varieties of the same dimension, with $\widetilde{\mathcal{C}}$ irreducible and $\mathcal{C}\subset\widetilde{\mathcal{C}}$, so we get $\widetilde{\mathcal{C}}=\mathcal{C}$ by \cite[Proposition 9.4.10]{CLO}. It follows that $\mathcal{C}$ is an irreducible component of $V\times \textrm{Mat}_{n\times n}$ and of $\mathcal{C}_n^3$.  Take an arbitrary pair $(A,B)\in\mathcal{C}$.  Thus there exist non-scalar matrices $C$ and $D$ such that $A\leftrightarrow C\leftrightarrow D \leftrightarrow B$.  The centralizer of any matrix contains a non-derogatory matrix, as if $J=J_{m_1}(\lambda_1)\oplus\cdots\oplus J_{m_t}(\lambda_t)$ is a matrix in the Jordan canonical form, then $J_{m_1}(\mu_1)\oplus\cdots\oplus J_{m_t}(\mu_t)$ commutes with $J$ and it is non-derogatory if the eigenvalues $\mu_1,\ldots ,\mu_t$ are pair-wise distinct. We can therefore choose a non-derogatory matrix $E$ commuting with $C$.  It follows that $(A+\lambda E,B)\in\mathcal{C}_n^3$ for every $\lambda\in k$.  In Lemma \ref{lemma:derogatory_variety} we have shown that the set of non-derogatory matrices is open in $\textrm{Mat}_{n\times n}$ in the Zariski topology. As $E$ is non-derogatory, it therefore follows that $A+\lambda E$ is non-derogatory  for all but finitely many scalars $\lambda\in k$. For each such $\lambda$ the pair $(A+\lambda E,B)$ does not belong to $\mathcal{C}$.  Therefore there exists another irreducible component $\mathcal{C}'$ of $\mathcal{C}_n^3$, different from $\mathcal{C}$, such that $(A+\lambda E,B)\in \mathcal{C}'$ for infinitely many $\lambda\in k$.  In the Zariski topology the closure of an infinite number of points on a line is equal to the whole line. As $\mathcal{C}'$ is closed, it therefore has to contain $(A+\lambda E,B)$ for each $\lambda \in k$. In particular $(A,B)\in \mathcal{C}'$, so $\mathcal{C}\subsetneq \mathcal{C}'$.  This is a contradiction to $\mathcal{C}$ being a component of $\mathcal{C}_{n}^3$.
\end{proof}

We close with a brief study of the distance-$3$ commuting set over the field of real numbers.  We start with the following example.

\begin{example}\label{example:counterexample_4x4} Let $A=\left( \begin{smallmatrix} 
1 & -1 &0& 3\\
-1 & 1 &0& -1\\
-2 & 2& 0& -4\\
0 &  0& 0& -2
\end{smallmatrix} \right),  
B =\left( \begin{smallmatrix} 
1&1 &0 & 0\\
-1& 1 &0 & 0\\
0 & 0 &-1 &1\\
0 & 0 &-1 &-1\end{smallmatrix} \right)\in \mathrm{Mat}_{4\times 4}(\mathbb{R}).$ We claim that $d_{\mathbb{R}}(A,B)=4$ and $d_{\mathbb{C}}(A,B)=3.$  Indeed, $A$ is of rank 2 with the eigenvalues $\{2,-2,0,0\}$, so it is  derogatory, and by Theorem \ref{theorem:equivalences} $d_{\mathbb{C}}(A,B)\leq 3$.

Assume now that $ A \leftrightarrow C \leftrightarrow D \leftrightarrow B$ for some $C,D \in \mathrm{Mat}_{4\times 4}(\mathbb{R})$. Then direct computations show that $C$ is of the form
$\left( \begin{smallmatrix} a &        b  &     c   &                       d\\
b+2c & a-2c & c &                          b+c\\
e &        f  &      a-b-2c+\frac e2-\frac  f2 & b+c-d+\frac  e2+\frac f2\\
0   &      0    &   0            &              a-c-d\end{smallmatrix} \right)$
and $D$ is of the form
$\left( \begin{smallmatrix} a' &  b' & 0 & 0\\
-b' &a' & 0 & 0\\
0 &  0 &  c' & d'\\
0 &  0 & -d' &c'\end{smallmatrix} \right).$
It can be  verified straightforwardly  that no two non-scalar such real matrices $C$ and $D$ commute.  Thus $d_\mathbb{R}(A,B)\geq 4$.  It follows that $d_\mathbb{C}(A,B)=3$ and $d_\mathbb{R}(A,B)=4$. 
\end{example}

This example, combined with the following lemma, implies Proposition \ref{prop:c34r}, which states that $\mathcal{C}_4^3(\mathbb{R})$ is \emph{not} a variety.

\begin{lemma}\label{lemma:nonvariety}
Let $k$ be an infinite field, and suppose that there exist $n\times n$ matrices $A$ and $B$ with $d_k(A,B)=4$ where $A$ has rank at most $n-2$.  Then $\mathcal{C}_n^3(k)$ is not a variety.
\end{lemma}

\begin{proof}   Suppose for the sake of contradiction that $\mathcal{C}_n^3(k)$ is a variety.  Then so too is its intersection with any other variety.  Let $V$ be the variety obtained by intersecting $\mathcal{C}_n^3(k)$ with the codimension-$n^2$ linear space $L$ defined by fixing the first $n^2$ coordinates to be our given matrix $A$.  Identifying $L$ with $k^{n^2}$, we can think of $V$ as a variety in $n^2$-dimensional affine space.  By our hypotheses, it is a proper subvariety; for instance, it does not contain the point $B$.

By the argument from the proof of Proposition \ref{prop:d_geq_4}, the set of $n\times n$ matrices over an infinite field $k$ that have an eigenvalue in $k$ forms a Zariski-dense set.  Let $B'$ be such a matrix.  Then $B'$ commutes with some rank $1$ matrix $D$.  Since $A$ is an $n\times n$ matrix of rank at most $n-2$, there exists a rank $1$ matrix $C=uv^T$ that commutes with both $A$ and $D$: simply choose $u$ to be orthogonal to the row space of $A$ and the row space of $D$, and choose $v$ to be orthogonal to the column space of $A$ and the column space of $D$.  Thus we have a commuting chain $A\leftrightarrow C\leftrightarrow D\leftrightarrow B'$, where both $C$ and $D$ have rank $1$ and are therefore nonscalar.  Thus $d_\mathbb{R}(A,B')\leq 3$.  Since $A$ is distance at most $3$ from all matrices in a Zariski-dense set, we have that $V$ is Zariski-dense in in $L$.  This is a contradiction to the fact that it is a proper subvariety.
\end{proof}

It is currently an open problem whether $\mathcal{C}_n^3(\mathbb{R})$ is a variety for $n=3$ and for $n\geq 5$.  One possible tool for approaching this is the following result.

\begin{proposition}\label{theorem:Rdistance3} Let $k$ be a field and $n\ge 3$ such that $d_k(A,B)=d_{\overline{k}}(A,B)$ for all $(A,B)\in\textrm{Mat}_{n\times n}(k)^2$. The distance-$3$ commuting set $\mathcal{C}_n^3(k)$ is an affine variety.
\end{proposition}

\begin{proof}
From our assumption on $d_k$ and $d_{\overline{k}}$, we have  that
\[\mathcal{C}_n^3(k)=\mathcal{C}_n^3(\overline{k})\cap \textrm{Mat}_{n\times n}(k)^2.\]
By Lemma \ref{lemma:realclosed}, if $k$ is a field and $V\subset {\overline{k}}^m$ is an affine variety, then so is $V\cap k^m$.  We know by Proposition \ref{theorem:distance3} that $\mathcal{C}_n^3(\overline{k})$ is an affine variety, so its intersection with the affine space $\textrm{Mat}^2_{n\times n}(k)$ is as well.  We conclude that $\mathcal{C}_n^3(k)$ is an affine variety.
\end{proof}

We remark that if $d_{\overline{k}}(A,B)\le 2$, we have $d_k(A,B)=d_{\overline{k}}(A,B)$.  This follows by definition if $d_{\overline{k}}(A,B)\leq 1$.  We also have $d_{\overline{k}}(A,B)\le 2$ if and only if the rank of the matrix $M_{A,B}$ from Section \ref{section:distance2} is at most $n^2-2$; since rank is independent of field, this is equivalent to $d_{k}(A,B)\le 2$. It follows that $d_{\overline{k}}(A,B)=2$ if and only if $d_k(A,B)=2$.

Thus the only way to have  $d_k(A,B)\neq d_{\overline{k}}(A,B)$ is with $d_k(A,B)\geq 4$ and $d_k(A,B)>d_{\overline{k}}(A,B)\geq 3$.  This can indeed happen, as illustrated in Example \ref{example:counterexample_4x4}. At present we do not know of any instances of $k\neq \overline{k}$ where $d_k(A,B) =  d_{\overline{k}}(A,B)$ for all $A,B\in \textrm{Mat}^2_{n\times n}(k)$ at least for some fixed $n>2$.  The following example shows that we do not always have $d_\mathbb{R}(A,B)=d_{\mathbb{C}}(A,B)$ for the case of $3\times 3$ matrices.  This does not immediately imply that  $\mathcal{C}_3^3(\mathbb{R})$ is not a variety; rather, it indicates that any proof that it is a variety cannot rely on the previous proposition.

\begin{example}\label{example:counterexample_3x3} Let $A= \begin{pmatrix} 
0 & 1 & 0\\
-1 & 0 & 0\\
0 & 0 & 0
\end{pmatrix},  
B =\begin{pmatrix} 
0 & -1 & 0\\
1 & 0 & 1\\
0 & 0 & 0
\end{pmatrix}\in \mathrm{Mat}_{3\times 3}(\mathbb{R}).$ We claim that $d_{\mathbb{R}}(A,B)=4$ and $d_{\mathbb{C}}(A,B)=3.$ Assume that $ A \leftrightarrow C \leftrightarrow D \leftrightarrow B$ for some $C$ and $D$. Then direct computations show that $C$ is of the form
$\begin{pmatrix}
a & b & 0\\
-b & a & 0\\
0 & 0  & c\end{pmatrix}$
and $D$ is of the form
$\begin{pmatrix}
a' &  b' & c'\\
-b' & a' & -b'\\
0 &  0 &  a'-c'\end{pmatrix}.$
A straightforward calculation shows that no two non-scalar such real matrices $C$ and $D$ commute, so $d_\mathbb{R}(A,B)= 4$. On the other hand, the complex matrices of this form
$C=\begin{pmatrix}
0 & 1 & 0\\
-1 & 0 & 0\\
0 & 0 & i
\end{pmatrix}$ and $D=\begin{pmatrix}
0 & 1 & i\\
-1 & 0 & -1\\
0 & 0 & -i
\end{pmatrix}$ commute, which shows that $d_\mathbb{C}(A,B)=3$.
\end{example}



\medskip

\noindent \textbf{Acknowledgements.}  The first author was supported by the Clare Boothe Luce Program. The investigations of the second author are financially supported by RSF grant 17-11-01124.  The fourth author is partially supported by Slovenian Research Agency (ARRS) grants N1-0103 and P1-0222. The authors would like to thank Ngoc Tran, who organised the discussions that led to this project; as well as Spencer Backman and Martin Ulirsch, who took part in those discussions.  They also thank Jan Draisma, Tom Garrity, Jake Levinson, and Eyal Markman for many helpful discussions and ideas about higher distance commuting varieties.

 \bibliographystyle{abbrv}

\end{document}